\documentclass[]{amsart}

\usepackage[utf8]{inputenc}
\usepackage[toc]{appendix}

\usepackage[english]{babel}

\usepackage{graphicx}
\usepackage{subfigure}
\usepackage[justification=centering]{caption}

\usepackage{amsmath,a4wide}
\usepackage{amsfonts}
\usepackage{amssymb}
\usepackage{amsthm}
\usepackage{marvosym}

\usepackage{enumerate}
\usepackage{twoopt}

\theoremstyle{plain}
\newtheorem{theorem}{Theorem}[section]

\theoremstyle{plain}
\newtheorem{corollary}[theorem]{Corollary}

\theoremstyle{plain}
\newtheorem{lemma}[theorem]{Lemma}

\theoremstyle{plain}
\newtheorem{proposition}[theorem]{Proposition}

\theoremstyle{definition}
\newtheorem{definition}[theorem]{Definition}

\theoremstyle{remark}
\newtheorem*{remark}{Remark}

\theoremstyle{remark}

\theoremstyle{definition}
\newtheorem{example}{Example}[section]

\providecommand{\norm}[1]{\lVert#1\rVert}

\newcommand{\R}{\mathbb{R}}
\newcommand{\C}{\mathbb{C}}
\newcommand{\Rd}{\mathbb{R}^d}
\newcommand{\Z}{\mathbb{Z}}

\newcommand{\Lt}[1][d]{L^2(\R^{#1})}
\newcommand{\G}{\mathcal{G}}
\newcommand{\F}{\mathcal{F}}
\newcommand{\B}{\mathcal{B}}

\renewcommand{\l}{\lambda}
\renewcommand{\L}{\Lambda}

\newcommand{\J}{
	\left(
		\begin{array}{cc}
		0 & I\\
		-I & 0
		\end{array}
	\right)
}

\renewcommand{\S}{
	\left(
		\begin{array}{cc}
		A & B\\
		C & D
		\end{array}
	\right)
}
\newcommand{\iS}{
	\left(
		\begin{array}{cc}
		D^T & -B^T\\
		-C^T & A^T
		\end{array}
	\right)
}

\newcommand{\VP}{
	\left(
		\begin{array}{cc}
		I & 0\\
		-P & I
		\end{array}
	\right)
}

\newcommand{\ML}{
	\left(
		\begin{array}{cc}
		L^{-1} & 0\\
		0 & L^T
		\end{array}
	\right)
}

\newcommand{\SW}{
	\left(
		\begin{array}{cc}
		L^{-1}Q & L^{-1}\\
		PL^{-1}Q-L^T & PL^{-1}
		\end{array}
	\right)
}

\newcommand{\RotMat}[1][\tau]{
	\left(
		\begin{array}{rr}
		\cos #1 & \sin #1\\
		-\sin #1 & \cos #1
		\end{array}
	\right)
}

\newcommandtwoopt{\HypMat}[2][\cosh(\tau)][\sinh(\tau)]{
	\left(
		\begin{array}{cc}
		#1 & #2\\
		#2 & #1
		\end{array}
	\right)
}

\newcommandtwoopt{\Vector}[2][x][\omega]{
	\left(
		\begin{array}{cc}
		#1 & #2\\
		\end{array}
	\right)
}

\newcommandtwoopt{\VectorT}[2][x][\omega]{
	\left(
		\begin{array}{cc}
		#1 \\
		#2
		\end{array}
	\right)
}

\newcommand{\modFT}{\widehat{J}}

\newcommand{\chirp}[1][-P]{\widehat{V_{#1}}}

\newcommand{\rescaleop}[1][L]{\widehat{M_{#1}}}

\newcommand{\metaop}[1][W]{\widehat{S_{#1}}}

\newcommandtwoopt{\go}[2][t^2][1]{2^{#2/4} e^{-\pi #1}}

\title[Gabor Frame Sets of Invariance]{Gabor Frame Sets of Invariance - A Hamiltonian Approach to Gabor Frame Deformations}
\author[M.\ Faulhuber]{Markus Faulhuber}
\address{Markus Faulhuber: NuHAG, Faculty of Mathematics, University of Vienna, Oskar-Morgenstern-Platz 1, 1090 Vienna, Austria}
\email{markus.faulhuber@univie.ac.at}
\thanks{Austrian Science Fund (FWF): [P26273-N25]}
\date{}
\numberwithin{equation}{section}

\begin{document}

\begin{abstract}
	In this work we study families of pairs of window functions and lattices which lead to Gabor frames which all possess the same frame bounds. To be more precise, for every generalized Gaussian $g$, we will construct an uncountable family of lattices $\lbrace \Lambda_\tau \rbrace$ such that each pairing of $g$ with some $\Lambda_\tau$ yields a Gabor frame, and all pairings yield the same frame bounds. On the other hand, for each lattice we will find a countable family of generalized Gaussians $\lbrace g_i \rbrace$ such that each pairing leaves the frame bounds invariant. Therefore, we are tempted to speak about \textit{Gabor Frame Sets of Invariance}.
\end{abstract}


\maketitle

\section{Introduction and Notation}

A \textit{Gabor frame} (or \textit{Weyl-Heisenberg frame}) for $\Lt$ is generated by a (fixed, non-zero) \textit{window function} $g \in \Lt$ and an index set $\L \subset \R^{2d}$. It is denoted by $\mathcal{G}(g,\L)$ and consists of \textit{time-frequency shifted} versions of $g$.

We denote by $\l = (x, \omega) \in \Rd \times \Rd$ a point in the \textit{time-frequency plane} and use the following notation for a \textit{time-frequency shift} by $\l$:
\begin{equation}\label{eq_TFShift}
	\pi(\l)g(t) = M_\omega T_x \, g(t) = e^{2 \pi i \omega \cdot t} g(t-x), \quad x,\omega,t \in \Rd.
\end{equation}
The operators involved in Equation \eqref{eq_TFShift} are the \textit{translation operator}
\begin{equation}
	T_x g(t) = g(t-x)
\end{equation}
and the \textit{modulation operator}
\begin{equation}
	M_\omega g(t) = e^{2 \pi i \omega \cdot t}.
\end{equation}
The latter one shifts a function in the Fourier or frequency space, hence the name time-frequency shift for the composition of the mentioned operators.
\begin{remark}
	The point `` $\cdot$ " denotes the Euclidean inner product of two column vectors, i.e.\ $\omega \cdot t = \langle \omega , t \rangle = \omega^T t$ (e.g.\ in Equation \eqref{eq_TFShift}). Also, we will use the notation $x^2 = x \cdot x$, $Sx \cdot y = x^T S^T y$ and $S x^2 = x^T S^T x$ for $x,y \in \Rd$ and $S$ a $d \times d$ matrix.
\end{remark}
\begin{remark}
	We note that the translation and modulation operator do not commute in general, in fact
	\begin{equation}
		M_\omega T_x = e^{2 \pi i \omega \cdot x} T_x M_\omega.
	\end{equation}
	This formula is closely related to the commutation relations in quantum mechanics.
\end{remark}

Let $\L$ be an index set. The time-frequency shifted versions of the window $g$ with respect to $\L$ are called \textit{atoms} and the set
\begin{equation}
	\G(g,\L) = \lbrace \pi(\lambda)g \, | \, \lambda \in \Lambda \rbrace
\end{equation}
is called a \textit{Gabor system}. $\mathcal{G}(g,\L)$ is called a frame if it fulfills the \textit{frame property}
\begin{equation}\label{Frame}
	A \norm{f}_2^2 \leq \sum_{\l \in \L} \left| \langle f, \pi(\l) g \rangle \right|^2 \leq B \norm{f}_2^2, \quad \forall f \in \Lt
\end{equation}
for some positive constants $0< A \leq B < \infty$ called \textit{frame constants} or \textit{frame bounds}. The index set $\L \subset \R^{2d}$ is called a \textit{lattice} in the time-frequency plane if and only if there exists an invertible (non-unique) $2d \times 2d$ matrix $S$, in the sense that $\L = S \Z^{2d}$. The volume of the lattice is defined as
\begin{equation}\label{Volume}
	vol(\L) = |\det(S)|
\end{equation}
and the \textit{density} of the lattice is given by
\begin{equation}
	\delta(\Lambda) = \frac{1}{vol(\Lambda)}.
\end{equation}

For more details on frames, Gabor frames and time-frequency analysis we refer to \cite{Chr03}, \cite{FeiGro97}, \cite{FeiLue12}, \cite{Gro01}, \cite{Hei06}.

\begin{remark}
	The reader familiar with the topic of time-frequency analysis will note that we restricted ourselves to the Hilbert space case. The spaces usually involved when it comes to time-frequency analysis are the \textit{modulation spaces}
	\begin{equation}
		M^p(\R^d) = \lbrace f \in \mathcal{S}'(\Rd) \, | \, \norm{f}_{M^p}^p < \infty \rbrace,
	\end{equation}
	where
	\begin{equation}
		\norm{f}_{M^p}^p = \int_{\Rd} \int_{\Rd} | \langle f,M_\omega T_x g \rangle |^p \, dx \, d\omega,
	\end{equation}
	with (non-zero) fixed $g \in \mathcal{S}(\Rd)$ (see e.g\ \cite{Gro01}, \cite{Gro14}). We note that $M^p(\Rd)$ is independent of the choice of $g \in \mathcal{S}(\Rd)$ and furthermore, $M^2(\Rd) = \Lt$. In time-frequency analysis the spaces $M^1(\R^d)$ and its dual space $M^\infty(\R^d)$ replace in a natural way the Schwartz space $\mathcal{S}(\R^d)$ and its dual space, the space of tempered distributions $\mathcal{S}'(\R^d)$ usually used in the field of analysis. In time-frequency analysis the test or window functions are often assumed to be in \textit{Feichtinger's Algebra} $\mathcal{S}_0(\R^d) = M^1(\R^d)$, whose dual space $\mathcal{S}'_0(\R^d) = M^\infty(\R^d)$ is the natural space of distributions in time-frequency analysis. For further reading on the topic of modulation spaces we refer to \cite{Fei81}, \cite{Gro01}, \cite{Gro14}.
\end{remark}

From Equation \eqref{Frame} we see that the frame bounds $A$ and $B$ depend on the window $g$ and the lattice $\Lambda$. Fixing the density of the lattice, a question arising is how far one can deform the window or the lattice or both without destroying the frame property. Results in this direction are usually called perturbation or deformation results and some are given in \cite{Gos14}, \cite{FeiKai04}, \cite{Gro14}, \cite{GroOrtRom13}. Usually, results concerning deformations of Gabor frames do not describe the behavior of the frame bounds explicitly, but rather state whether the frame property is kept at all or not. We will present some perturbation results where not only the frame property is kept, but also the frame bounds.

Although many of the definitions and well-known, general theorems are stated for $\Lt$, we will state our results only for $d=1$.

This work is structured as follows. In Section \ref{SympGroup} we recall the basic properties of the symplectic group, in Section \ref{MetaGroup} we recall the corresponding properties of the metaplectic group as well as the interplay between the two mentioned groups. Finally, in Section \ref{Sets_of_Invariance} we will state and proof the main theorem of this work and present some examples.

\section{The Symplectic Group}\label{SympGroup}

As we want to describe a lattice by a matrix one possible way of describing the deformation process is by multiplying the generating matrix with another matrix from the left. In particular, all our deformations will by carried out by symplectic matrices. Hence, in this section we want to recall some basic facts about the symplectic group and its elements, the symplectic matrices. They are widely used in Hamiltonian mechanics and also serve as tools in time-frequency analysis \cite{Gos11}, \cite{Fol89}, \cite{Gro01}. All results of this section can be found in full generality in de Gosson's book \cite{Gos11}. We will also list further references at many points. Although we will only treat the 1-dimensional case, most results can be formulated verbatim for $d>1$ which is why we formulate them for the latter case.

\subsection{Symplectic Matrices}
\begin{definition}\label{definition_symplectic_matrix}
	A matrix $S \in GL(2d,\R)$ is called \textit{symplectic} if and only if
	\begin{equation}\label{eq_symplectic_matrix}
		S J S^T = S^T J S = J,
	\end{equation}
	where $J = \J$, $0$ is the $d \times d$ zero matrix and $I$ is the $d \times d$ identity matrix. $J$ is called the \textit{standard symplectic matrix}.
\end{definition}

\begin{remark}
	We note that condition \eqref{eq_symplectic_matrix} is redundant. Actually, we have
	\begin{equation}
		S J S^T = J \Longleftrightarrow S^T J S = J.
	\end{equation}
	From \eqref{eq_symplectic_matrix} we conclude that all symplectic matrices $S \in Sp(2d,\R)$ must have determinant equal to $\pm 1$. In fact, if $S \in Sp(2d,\R)$ then $det(S) = 1$, see \cite{Gos11}, \cite{Gos13}, \cite{Gro01}. Also, $Sp(2d,\R)$ is a subgroup of $SL(2d,\R)$ and in the case $d = 1$ we have $Sp(2,\R) = SL(2,\R)$. In all other cases where $d > 1$, $Sp(2d,\R)$ is a proper subgroup of $SL(2d,\R)$.
\end{remark}

\begin{lemma}
	The set of all symplectic matrices forms a group denoted by $Sp(2d,\R)$.
\end{lemma}

\begin{proof}
	Let $S_1, \, S_2 \in Sp(2d,\R)$. It follows from Equation \eqref{eq_symplectic_matrix} that the product $S_1 S_2 \in Sp(2d,\R)$. Taking the inverse of the double equality in \eqref{eq_symplectic_matrix} and using the fact that $J^{-1} = -J$ we see that $S^{-1} \in Sp(2d,\R)$ if $S \in Sp(2d,\R)$.
\end{proof}

It is convenient and commonly used write symplectic matrices as block matrices in the following form
\begin{equation}
	S = \S,
\end{equation}
where $A,B,C,D$ are $d \times d$ matrices. With this notation we have the following formula for the inverse of a symplectic matrix
\begin{equation}\label{eq_inverse_symplectic_matrix}
	S^{-1} = \iS.
\end{equation}
In the case $d = 1$ this reduces to the well-known inversion formula for a matrix $S$ belonging to $SL(2,\R)$, as $A,B,C,D \in \R$ are scalars.

\subsection{Free Symplectic Matrices}

We will now introduce the building blocks of the symplectic group, the free symplectic matrices and state that any symplectic matrix is the product of these building blocks \cite{Gos11}, \cite{GosLue14}.

\begin{definition}\label{definition_free_symplectic_matrix}
	We call a symplectic matrix $S = \S \in Sp(2d,\R)$ a \textit{free symplectic matrix} if $B$ is invertible.
\end{definition}

\begin{definition}\label{definition_generator_matrix}
	Let $P$ be a symmetric $d \times d$ matrix and let $L$ be an invertible $d \times d$ matrix. We define the following $2d \times 2d$ matrices
	\begin{align}
		J & = \J\\
		V_P & = \VP \label{eq_lower_tri_matrix}\\
		M_L & = \ML \label{eq_dilation_matrix}
	\end{align}
	which all belong to $Sp(2d,\R)$. We call them \textit{generator matrices} for the free symplectic matrices.
\end{definition}

The name generator matrix is justified by the following propositions.

\begin{proposition}\label{proposition_symplectic_factorization}
	With the notation of Definition \ref{definition_generator_matrix} we get that any free symplectic matrix $S = \S$ can be factored as
	\begin{equation}\label{eq_factor_symplectic_matrix_1}
		S = V_{-DB^{-1}}M_{B^{-1}}JV_{-B^{-1}A}.
	\end{equation}
\end{proposition}

A proof is given in \cite{Gos11} or \cite{Fol89}. It makes use of well-known factorization results and properties of symplectic matrices.

\subsection{Generating Functions}

Following \cite{Gos11} we will point out connections between quadratic forms in $(x,x')$ and free symplectic matrices. The motivation comes from \textit{Hamiltonian mechanics}. We want to describe the motion of a particle depending on two variables usually called position ($x$) and momentum ($p$) which depend on time ($t$) and are coupled by \textit{Hamilton's equations}
\begin{equation}
	\begin{aligned}
		& \dot{x}(t) = \frac{\partial}{\partial p}H(x(t),p(t))\\
		& \dot{p}(t) = -\frac{\partial}{\partial x}H(x(t),p(t)).
	\end{aligned}
\end{equation}
Here, $H(x(t),p(t))$ is the \textit{Hamiltonian} or \textit{Hamilton function}. For more details on Hamiltonian mechanics see \cite{Arn89}.

Given two different positions $x$ and $x'$ of a particle we want to know the initial and final momentum $p$ and $p'$ assuming that the motion is linear, meaning we have the linear system $(x,p) = S(x',p')$. This is equivalent to
\begin{equation}
	\begin{aligned}
		& x = A x' + B p'\\
		& p = C x' + D p'.
	\end{aligned}
\end{equation}
In order to solve this system of equations for $(p,p')$, clearly $B$ has to be invertible.

In the case of time-frequency analysis the proper way to use and interpret Hamiltonian mechanics is by replacing position by time and momentum by frequency. The following proposition is again formulated in the context of time-frequency analysis.
\begin{proposition}\label{proposition_quad_form_symplectic_matrix_equivalence}
	Let $S = \S \in Sp(2d,\R)$ be a free symplectic matrix. Let $P,Q$ be $d \times d$ symmetric matrices and let $L$ be a $d \times d$ invertible matrix.
	\begin{enumerate}[(i)]
		\item Then we have
		\begin{equation}\label{eq_quad_form_symplectic_matrix}
			(x,\omega) = S(x',\omega') \Longleftrightarrow
			\begin{cases}
				\omega = \partial_x W(x,x'),\\
				\omega' = -\partial_{x'} W(x,x')
			\end{cases}
		\end{equation}
		where $W$ is the quadratic form
		\begin{equation}\label{eq_quad_form_ABCD}
			W(x,x') = \frac{1}{2}DB^{-1}x^2-B^{-1}x \cdot x' + \frac{1}{2} B^{-1}A {x'}^2
		\end{equation}
		where $DB^{-1}$ and $B^{-1}A$ are symmetric.
		\item To every quadratic form
		\begin{equation}\label{eq_quad_form_PQL}
			W(x,x') = \frac{1}{2} Px^2 - L x \cdot x' + \frac{1}{2} Q {x'}^2
		\end{equation}
		we can associate the free symplectic matrix
		\begin{equation}\label{eq_associate_symplectic_matrix_SW}
			S_W = \SW.
		\end{equation}
		We call the quadratic form in \eqref{eq_quad_form_PQL} the generating function of $S_W$ in \eqref{eq_associate_symplectic_matrix_SW}
	\end{enumerate}
\end{proposition}

\begin{remark}
	Note the connection between the generating function and the factorization of a free symplectic matrix.
\end{remark}

\begin{theorem}\label{thm_factorization_smyplectic_matrix}
	For every $S \in Sp(2d,\R)$ there exist two (non-unique) free symplectic matrices $S_W$ and $S_{W'}$ such that $S = S_W S_{W'}$.
\end{theorem}

\begin{corollary}\label{cor_generator_matrix}
	The set of all matrices
	\begin{equation}
		\lbrace V_P, M_L, J \rbrace
	\end{equation}
	generates the symplectic group $Sp(2d,\R)$.
\end{corollary}

\section{The Metaplectic Group}\label{MetaGroup}

The second way to perform a deformation of a Gabor frames is to perturb the window. We will describe this process by letting some unitary operator act on the window. In fact, we will only deal with some special operators called metaplectic.

The metaplectic group and its elements, the metaplectic operators, are widely used in quantum mechanics and in time-frequency analysis. There is a close connection to the symplectic group and this interplay might be used to solve problems in quantum mechanics once the solution for the corresponding classical problem is known \cite{Gos11}. In time-frequency analysis this property can be used to deform Gabor frames without destroying their frame property and even keeping the optimal frame bounds \cite{Gos14}, \cite{Gro01}.

Again, since there is not much difference between formulating the results for $d=1$ and $d>1$, we state them in full generality, although we will only consider the case $d=1$ later on. The results can be found in \cite{Gos11} in full detail.

\subsection{The Group $Mp(2d,\mathbb{R})$}

\begin{definition}\label{definition_metaplectic_group_exact}
	The \textit{metaplectic group} $Mp(2d,\R)$ is the connected two-fold cover of the symplectic group $Sp(2d,\R)$. Equivalently, we can define $Mp(2d,\R)$ by saying that the sequence
	\begin{equation}
		0 \rightarrow \Z_2 \rightarrow Mp(2d,\R) \rightarrow Sp(2d,\R) \rightarrow 0
	\end{equation}
	is exact.
\end{definition}

\begin{remark}
	A sequence
	\begin{equation}
		A_0 \rightarrow A_1 \rightarrow \dots \rightarrow A_n \rightarrow A_{n+1}
	\end{equation}
	of morphisms is called \textit{exact}, if the image of each morphism is equal to the kernel of the next morphism
	\begin{equation}
		im(A_{k-1} \rightarrow A_k) = ker(A_k \rightarrow A_{k+1}), \quad k = 1, \dots , n.
	\end{equation}
\end{remark}

We want to use a more constructive approach to define the metaplectic group.

\subsection{Metaplectic Operators and the Quadratic Fourier Transform}

The metaplectic group is a group of unitary operators on $\Lt$ \cite{Gos07}, \cite{Gos11}, \cite{Gro01}, \cite{Rei89}. Let $\psi \in \mathcal{S}(\R^d)$ be a function in the Schwartz space. Following de Gosson \cite{Gos11} we define the following operators.
\begin{itemize}
	\item The modified Fourier transform $\modFT$ defined by
	\begin{equation}\label{eq_modified_FT}
		\modFT \psi(t) = i^{-d/2} \int_{\R^d} \psi(t') \, e^{-2 \pi i \, t \cdot t'} \, dt'.
	\end{equation}
	\item The linear ``chirps"
	\begin{equation}\label{eq_chirp}
		\chirp \psi(t) = e^{\pi i \, Pt \cdot t} \psi(t)
	\end{equation}
	with $P$ being a real, symmetric $d \times d$ matrix.
	\item The rescaling operator
	\begin{equation}\label{eq_rescale_operator}
		\rescaleop[L,n] \psi(t) = i^n \sqrt{|det(L)|} \psi(Lt),
	\end{equation}
	where $L$ is invertible and $n$ is an integer corresponding to a choice of $arg(det(L))$, to be more precise
	\begin{equation}\label{eq_Maslov_index}
		n \pi \equiv arg(det(L)) \quad \text{mod } 2 \pi.
	\end{equation}
\end{itemize}
The class modulo $4$ of the integer $n$ appearing in the definition of the rescaling operator \eqref{eq_rescale_operator} is called ``Maslov index" \cite{Gos11}, \cite{GosLue14}.

As in the section on the symplectic group we will associate quadratic forms to metaplectic operators and we will also see the interplay between the symplectic and the metaplectic group.

\begin{definition}\label{definition_quadratic_FT}
	Let $S_W$ be the free symplectic matrix
	\begin{equation}
		S_W = \SW
	\end{equation}
	associated to the quadratic form $W(t,t') = \frac{1}{2} Pt^2 - L t \cdot t' + \frac{1}{2} Q {t'}^2$ (compare proposition \ref{proposition_quad_form_symplectic_matrix_equivalence} equations \eqref{eq_quad_form_PQL} and \eqref{eq_associate_symplectic_matrix_SW}).
	Let the operators $\modFT, \chirp$ and $\rescaleop[L,n]$ be defined as in \eqref{eq_modified_FT}, \eqref{eq_chirp} and \eqref{eq_rescale_operator} respectively. We call the operator
	\begin{equation}\label{eq_quadratic_FT}
		\metaop[W,n] = \chirp \rescaleop[L,n] \modFT \chirp[-Q]
	\end{equation}
	the \textit{quadratic Fourier transform} associated to the free symplectic matrix $S_W$.
\end{definition}

For $\psi \in \mathcal{S}(\R^d)$ we have the explicit formula
\begin{equation}\label{eq_quadratic_FT_formula}
	\metaop[W,n] \psi(t) = i^{n-\frac{d}{2}} \sqrt{|det(L)|} \int_{\R^d}  \psi(t') \, e^{2 \pi i \, W(t,t')} \, dt',
\end{equation}
where $W(t,t')$ is again the quadratic form as defined in \eqref{eq_quad_form_PQL} and Definition \ref{definition_quadratic_FT}.

\begin{remark}
	Although, all statements in this section were formulated for the Schwartz space $\mathcal{S}(\R^d)$, they also hold for Feichtinger's algebra $\mathcal{S}_0(\R^d)$ as well as for the Hilbert space $\Lt$ \cite{Gos14}.
\end{remark}

\begin{remark}
	We will frequently drop one or both of the indices $W$ and $n$ and will write $S$ instead of $S_W$ and $\widehat{S}$ or $\metaop$ instead of $\metaop[W,n]$. When the context allows, we will also use other indices than the ones mentioned.
\end{remark}

\begin{remark}
	As can be seen by formula \eqref{eq_quadratic_FT} a quadratic Fourier transform is a manipulation of a (suitable) function by a chirp, a modified Fourier transform, a dilation and another chirp. This is the exact same way in which the \textit{fractional Fourier transform} is described in \cite{Alm94} with an additional dilation in between the modified Fourier transform and the second chirp. Hence, the quadratic Fourier transform is an extension of the fractional Fourier transform in the sense that the directions in the time-frequency plane are scaled by some factor depending on the angle. For more details on the fractional Fourier transform see also \cite{GosLue14}.
\end{remark}

\begin{proposition}\label{proposition_inverse_meta}
	The operators $\metaop[W,n]$ extend to unitary operators on $\Lt$ and the inverse is given by
	\begin{equation}
		\metaop[W,n]^{-1} = \metaop[W^*,n^*],
	\end{equation}
	where $W^*(t,t') = -W(t',t)$ and $n^* = d-n$.
\end{proposition}

The fact that $\metaop[W,n]$ is a unitary operator is clear since, $\chirp$, $\rescaleop[L,n]$ and $\modFT$ are unitary. Obviously, we have
\begin{equation}
	\chirp^{-1} = \chirp[P], \qquad \rescaleop[L,n]^{-1} = \rescaleop[L^{-1},-n]
\end{equation}
and the inverse of the modified Fourier transform is given by
\begin{equation}
	\modFT^{-1} \psi(t) = i^{d/2} \int_{\R^{d}} \psi(t') \, e^{2 \pi i t \cdot t'} \, dt'
\end{equation}
We note that
\begin{equation}
	\modFT^{-1} \rescaleop[L^{-1},-n] = \rescaleop[-L^T,d-n] \modFT
\end{equation}
and hence,
\begin{equation}
	\metaop[W,n]^{-1} = \chirp[Q] \modFT^{-1} \rescaleop[L^{-1},-n] \chirp[P] = \metaop[W*,n*].
\end{equation}

\begin{definition}
	The subgroup of $\mathcal{U}(\Lt)$ generated by the quadratic Fourier transforms $\metaop[W,n]$ is called the metaplectic group and is denoted by $Mp(2d,\R)$. Its elements are called metaplectic operators.
\end{definition}

\begin{remark}
	To each quadratic form $W(t,t')$ we can actually associate not one but two metaplectic operators as, due to \eqref{eq_Maslov_index}, $\metaop[W,n]$ and $\metaop[W,n+2] = -\metaop[W,n]$ are equally good choices. This reflects the fact that the metaplectic operators are elements of the two-fold cover of the symplectic group.
\end{remark}

\begin{theorem}\label{thm_factorization_meta}
	For every $\widehat{S} \in Mp(2d,\R)$ there exist two quadratic Fourier transforms $\metaop[W,n]$ and $\metaop[W',n']$ such that $\widehat{S} = \metaop[W,n] \metaop[W',n']$.
\end{theorem}

The factorization in Theorem \ref{thm_factorization_meta} is not unique as the identity operator can always be written as $\metaop[W,n] \metaop[W^*,n^*]$.

\begin{corollary}
	The set of all operators
	\begin{equation}
		\lbrace \chirp, \rescaleop[L,n], \modFT \rbrace
	\end{equation}
	generates the metaplectic group.
\end{corollary}

Without further preparation we introduce the \textit{natural projection} of the metaplectic group $Mp(2d,\R)$ onto the symplectic group $Sp(2d,\R)$, which we will denote by $\pi^{Mp}$. For the details we refer to \cite{Gos11}.

\begin{theorem}\label{thm_natural_projection}
	The mapping
	\begin{equation}
		\begin{aligned}
			\pi^{Mp}: & & Mp(2d,\R) & \longrightarrow Sp(2d,\R)\\
			& & \metaop[W,n] & \longmapsto S_W
		\end{aligned}
	\end{equation}
	which to a quadratic Fourier transform associates a free symplectic matrix with generating function $W$, is a surjective group homomorphism. Hence,
	\begin{equation}
		\pi^{Mp} \left( \widehat{S} \widehat{S}' \right) = \pi^{Mp} \left( \widehat{S} \right) \pi^{Mp} \left( \widehat{S}' \right).
	\end{equation}
	and the kernel of $\pi^{Mp}$ is given by
	\begin{equation}
		ker(\pi^{Mp}) = \lbrace -I,+I \rbrace.
	\end{equation}
	hence, $\pi^{Mp}: Mp(2d,\R) \mapsto Sp(2d,\R)$ is a two-fold covering of the symplectic group.
\end{theorem}

\begin{definition}
	The mapping $\pi^{Mp}$ in Theorem \ref{thm_natural_projection} is called the natural projection of $Mp(2d,\R)$ onto $Sp(2d,\R)$.
\end{definition}

\begin{remark}
	The natural projections of the metaplectic generator elements are the symplectic generator elements.
	\begin{equation}
		\pi^{Mp}\left(\chirp[P]\right) = V_P, \quad \pi^{Mp}\left(\rescaleop[L,n]\right) = M_L, \quad \pi^{Mp}\left(\modFT\right) = J.
	\end{equation}
\end{remark}

\section{Gabor Frame Sets of Invariance}\label{Sets_of_Invariance}

We prepared the machinery of the symplectic and metaplectic group to the extend we need it in order to be able to deform Gabor frames without destroying their frame property. We are interested in Gabor frame deformations which leave the frame bounds invariant.

\begin{definition}
	Assume $\G(g,\Lambda)$ and $\G(g',\L')$ are Gabor frames with the same optimal frame bounds $A$ and $B$. We write
	\begin{equation}
		\G(g,\Lambda) \cong \G(g',\L').
	\end{equation}
\end{definition}

\begin{theorem}\label{FrameBounds}
	Let $\G(g,\Lambda)$ be a Gabor frame with optimal frame bounds $A$ and $B$. Let $\widehat{S} \in Mp(2d,\R)$ with projection $\pi^{Mp}(\widehat{S}) = S \in Sp(2d,\R)$. Then $\G(\widehat{S}g, S \Lambda)$ is also a Gabor frame and has the same optimal frame bounds $A$ and $B$.
\end{theorem}

A full proof is given in \cite{Gos14}.

\begin{remark}
	For any window a phase factor $c \in \C$ with $|c|=1$ is negligible in the sense that $\G(g,\Lambda)$ and $\G(c \, g, \Lambda)$ have the same frame bounds as can directly be seen from equation \eqref{Frame}.
\end{remark}

Theorem \ref{FrameBounds} is a particular case of the notion of \textit{Hamiltonian deformation of Gabor frames} (see \cite{Gos14}). It tells us under which conditions the frame property as well as the optimal frame bounds are kept when a Gabor frame suffers some disturbances. This is a very special case, as in general neither the optimal frame bounds nor the frame property might be kept under some general deformation of the frame. However, there are cases when the frame property might be kept without keeping the optimal frame bounds (see \cite{FeiKai04}, \cite{GroOrtRom13}). This is usually done by either deforming the window and fixing the lattice or the other way round. By Theorem \ref{FrameBounds} we know that these approaches are equivalent as long as we stick to symplectic and metaplectic deformations.

What we will see in the following sections is that it is possible to keep both, the frame property and the optimal frame bounds under certain lattice deformations, without changing the window. This is due to the fact that generalized Hermite functions, including the generalized Gaussians, are eigenfunctions with eigenvalues of modulus 1 of certain metaplectic operators. Hence, the corresponding symplectic matrix will deform the lattice, while the window can remain unchanged.

\subsection{Lattice Rotations and the Standard Gaussian}\label{sec_RotationInvariance}

From this point on, we will only consider the 1-dimensional case. The most popular 1-dimensional window function is probably the \textit{standard Gaussian} $g_1(t) = \go$. Although, Gabor frames with Gaussian window have been studied intensively, we still want to explore and exploit the Gabor family $\G(g_1,\Lambda)$ with $vol(\Lambda) < 1$ such that the frame property \eqref{Frame} is fulfilled \cite{Lyu92}, \cite{Sei92}.

One of the simplest manipulations of our Gabor frame is to rotate the lattice and calculate the corresponding window. This means that our lattice is deformed by the rotation matrix
\begin{equation}
	S_\tau = \RotMat
\end{equation}
and the corresponding deformation of the window is given by the action of the quadratic Fourier transform on the window $g_1$. To derive a formula for the resulting window we use Proposition \ref{proposition_quad_form_symplectic_matrix_equivalence} and Equation \eqref{eq_quadratic_FT_formula}.
\begin{equation}
	\metaop[\tau] \, g_1(t) = i^{n(\tau)-\frac{1}{2}} \sqrt{\frac{1}{|\sin(\tau)|}} \int_{\R} e^{2 \pi i \, W_\tau(t,t')} g_1(t') \, dt',
\end{equation}
where $n(\tau) \in \lbrace 0,1,2,3 \rbrace$ depends on $\tau$ and the choice of $arg\left(\sqrt{\sin(\tau)}\right)$ and where
\begin{equation}
	W_\tau(t,t') = \frac{1}{2 \sin(\tau)} \left((t^2+t'^2)\cos(\tau)-2tt')\right).
\end{equation}
This manipulation is meaningful whenever $ \tau \neq k \pi$, $k \in \Z$. Performing the calculations, we get
\begin{equation}
	\metaop[\tau]g_1(t) = 2^{1/4} i^{n(\tau)} e^{-i \frac{\tau}{2}} e^{-\pi t^2} = c \, g_1(t),
\end{equation}
with $|c| = 1$.
Hence, we have the result
\begin{equation}\label{eq_rotation_invariance}
	\G(g_1,\Lambda) \cong \G(g_1,S_\tau \Lambda),
\end{equation}
which means that the frame bounds of a Gabor frame with window $g_1$ stay invariant under rotation of the lattice.

Although, the ambiguity function of the standard Gaussian $g_1=\go$ is well-known and, though, it is an easy exercise to compute it, we will still do the calculations, as the procedure will be used extensively in a somewhat more general form in the rest of this work. For the definition and an interpretation of the ambiguity function see Appendix \ref{App_AmbiguityFunction}.

\begin{equation}
	\begin{aligned}
		Ag_1 (x, \omega) = & \int_\R g_1 \left(t + \frac{x}{2} \right) \overline{g_1 \left(t - \frac{x}{2} \right)} e^{-2 \pi i \omega t} \, dt\\
		= & \int_\R 2^{1/4} e^{-\pi (t + x/2)^2} 2^{1/4} e^{-\pi (t - x/2)^2} e^{-2 \pi i \omega t} \, dt\\
		= & \sqrt{2} \int_\R e^{-2 \pi (t^2 + x^2/4)} e^{-2 \pi i \omega t} \, dt\\
		= & \, \sqrt{2} \, e^{-\pi \frac{x^2}{2}} \int_\R \frac{1}{\sqrt{2}} e^{- \pi t^2} e^{-2\pi i \frac{\omega}{\sqrt{2}} t}\, dt\\
		= & \, e^{-\pi \frac{x^2}{2}} \, 2^{-1/4} \, \mathcal{F}{g_1} \left(\omega/\sqrt{2} \right)\\
		= & \, e^{-\pi \frac{x^2}{2}} \, 2^{-1/4} \, g_1 \left(\omega/\sqrt{2} \right)\\
		= & \, e^{-\pi \frac{x^2}{2}} e^{-\pi (w/\sqrt{2})^2}\\
		= & \, e^{-\frac{\pi}{2} \left(x^2 + \omega^2 \right)}
	\end{aligned}
\end{equation}

Here, $\F$ denotes the Fourier transform which is given by
\begin{equation}
	\F f(\omega) = \int_{\Rd} f(t) e^{-2 \pi i \omega \cdot t} \, dt.
\end{equation}
What we used in the calculations above are a change of variables and the Fourier invariance of the standard Gaussian, $\mathcal{F}(g_1) = g_1$. For the latter argument see \cite{Fol89}, \cite{Gro01}.

\subsection{Elliptic Deformations and Dilated Gaussians}

In Section \ref{sec_RotationInvariance} we saw that using the standard Gaussian window the Gabor frame bounds stay invariant under a rotation of the lattice. We will now generalize this result using ideas from Hamiltonian mechanics. For an introduction to Hamiltonian mechanics we refer to \cite{Arn89}. The rotation matrix
\begin{equation}
	S_\tau = \RotMat
\end{equation}
determines the \textit{flow} of the \textit{harmonic oscillator} with mass $m=1$ and resonance $\Omega = 1$. The Hamiltonian of this problem is given by
\begin{equation}\label{eq_Hamiltonian}
	H(x,\omega,\tau) = \frac{\omega^2}{2} + \frac{x^2}{2}
\end{equation}
and Hamilton's equations are given by
\begin{equation}
	\frac{d}{d \tau}\lambda = J \VectorT[\frac{\partial}{\partial x}H][\frac{\partial}{\partial \omega}H] = J \lambda,
\end{equation}
where $\lambda = (x,\omega)^T$ and both, $x$ and $\omega$ depend on $\tau$.

Written in its most general form the Hamiltonian of the harmonic oscillator is given by
\begin{equation}\label{eq_Hamilton_m_w0}
	H(x,\omega,\tau) = \frac{\omega^2}{2m} + \frac{m \Omega^2 x^2}{2},
\end{equation}
where $m$ is the mass of the particle and $\Omega$ is the resonance. The trajectories will be ellipses in standard position with semi-axis ratio $m \Omega$.

Assume, we are given the Gabor frame $\G(g_1,\Lambda)$ with standard Gaussian window and arbitrary lattice $\Lambda$, $vol(\Lambda) < 1$. Any dilation of the lattice by a symplectic matrix $M_{\sqrt{m}}$ can be compensated by a metaplectic dilation of the window such that the frame bounds remain unchanged, so
\begin{equation}
	\G(g_1,\Lambda) \cong \G(\rescaleop[\sqrt{m}] \, g_1, M_{\sqrt{m}} \Lambda).
\end{equation}
We compute that the dilated window is of the form
\begin{equation}\label{window_gw0}
	\rescaleop[\sqrt{m}] \, g_1(t) = c \, (2m)^{1/4} e^{-\pi m t^2} = g_{m}(t),
\end{equation}
where $|c|=1$.
Next, we compute the ambiguity function $Ag_{m}$.
\begin{equation}\label{eq_Ambiguity_generalized_Gaussian}
	\begin{aligned}
		Ag_{m}(x,\omega) & = \sqrt{2 m} \int_\R e^{-\pi m (t+x/2)^2} e^{-\pi m (t-x/2)^2} e^{-2 \pi i \omega t} \, dt\\
		& = \sqrt{2 m} \, e^{-\pi m x^2/2} \int_\R e^{-\pi m 2t^2} e^{-2 \pi i \omega t} \, dt\\
		& = e^{-\frac{\pi}{2} \left(m x^2 + \frac{\omega^2}{m}\right)}.
	\end{aligned}
\end{equation}
Hence, any level set of $Ag_{m}$ will be an ellipse in standard position with semi-axis ratio $m$. It is kind of self-evident to examine the behavior of $\G(g_{m},M_{\sqrt{m}} \Lambda)$ under deformations induced by the harmonic oscillator given by \eqref{eq_Hamilton_m_w0} with $\Omega = 1$. The flow is then determined by the symplectic matrix
\begin{equation}\label{eq_flow_w0}
	S_{\tau,m} = \left(
		\begin{array}{rr}
		\cos \tau & \frac{1}{m} \sin \tau\\
		-m \sin \tau & \cos \tau
		\end{array}
	\right).
\end{equation}
The corresponding metaplectic operator $\metaop[\tau,m]$ is determined by
\begin{equation}\label{eq_flow_w0_metaplectic}
	\metaop[\tau,m] \, g_{m} (t) = i^{n(\tau)-\frac{1}{2}} \sqrt{\frac{m}{|\sin \tau|}} \int_{\R} e^{2 \pi i \, W_{\tau,m}(t,t')} g_{m}(t') \, dt',
\end{equation}
where $n(\tau) \in \lbrace 0,1,2,3 \rbrace$ depends on the parameter $\tau$ as well as on the choice of $arg\left(\sqrt{\sin \tau}\right)$ and where
\begin{equation}
	W_{\tau,m}(t,t') = \frac{m}{2 \sin \tau} \left((t^2+t'^2)\cos \tau -2tt')\right)
\end{equation}
is the generating function of $S_{\tau, m}$. The verification of the upcoming formulas \eqref{eq_ellipse_invariance_window} and \eqref{eq_ellipse_invariance_ambiguity} can be found in appendix \ref{proof_ellipse_invariance}. We have
\begin{equation}\label{eq_ellipse_invariance_window}
	\metaop[\tau,m] \, g_{m} (t) = c \, g_{m}(t),
\end{equation}
with $|c|=1$. Hence, equation \eqref{eq_ellipse_invariance_window} implies that
\begin{equation}\label{eq_ellipse_invariance}
	\G(g_{m},\Lambda) \cong \G(g_{m}, S_{\tau, m} \Lambda),
\end{equation}
where $S_{\tau, m}$ is defined as in \eqref{eq_flow_w0}, as well as
\begin{equation}\label{eq_ellipse_invariance_ambiguity}
	Ag_{m}(x,\omega) = A \left(\metaop[\tau,m] \, g_{m} \right)(x,\omega),
\end{equation}
where $\metaop[\tau,m]$ is defined as given by equation \eqref{eq_flow_w0_metaplectic}. Summing up the results we get the following theorem.

\begin{theorem}\label{ellipse_invariance}
	Let $g_{m}(t) = c \, (2m)^{1/4} e^{-\pi m t^2}$ with $|c| = 1$ and let $\Lambda \subset \R^2$ be a lattice with $vol(\Lambda) < 1$
Let
	\begin{equation}
		S_{\tau,m} = \left(
			\begin{array}{rr}
			\cos \tau & \frac{1}{m} \sin \tau\\
			-m \sin \tau & \cos \tau
			\end{array}
		\right).
	\end{equation}
	be the deformation matrix acting on the lattice, then
	\begin{equation}\label{eq_ellipse_invariance_window_thm}
		\begin{aligned}
			\metaop[\tau,m] \, g_{m} (t) & = i^{n(\tau, m)-\frac{1}{2}} \sqrt{\frac{m}{|\sin \tau|}} \int_{\R} e^{2 \pi i \, W_{\tau,m}(t,t')} g_{m}(t') \, dt'\\
			& = c \, g_{m}(t),
		\end{aligned}
	\end{equation}
	with $|c|=1$ hence,
	\begin{equation}
		\G(g_{m},\Lambda) \cong \G(g_{m}, S_{\tau, m} \Lambda).
	\end{equation}
	Furthermore, for the ambiguity function $A\left(\metaop[\tau,m] \, g_{m}\right)$ we have
	\begin{equation}\label{eq_ellipse_invariance_ambiguity_thm}
		A \left( \metaop[\tau,m] \, g_{m} \right) (x,\omega) = Ag_{m}(x,\omega).
	\end{equation}
\end{theorem}

Setting $m = 1$, Theorem \ref{ellipse_invariance} implies that the Gabor frame bounds of a Gabor frame with standard Gaussian window stay invariant under a rotation of the lattice and hence, theorem \ref{ellipse_invariance} is a generalization of a result given in \cite{Gos14}.

Also, in Theorem \ref{ellipse_invariance} the symplectic geometry of the lattice remains unchanged, whereas the Euclidean geometry of the lattice will change in general. In the case where $m = 1$, meaning that we only rotate the lattice, the symplectic as well as the Euclidean geometric properties are kept.

In order to derive Theorem \ref{ellipse_invariance} we used a very geometric approach and a clear picture in mind about the flow induced by the harmonic oscillator. The crucial ingredient for theorem \ref{ellipse_invariance} to work is that we could easily and explicitly calculate the eigenfunctions of the metaplectic operator involved. We note that similar approaches have already been made in \cite{Dau88}, characterizing the (dilated) Hermite functions as eigenfunctions of certain localization operators. The geometric approach to phase space has also been used in \cite{DoeRom14} to construct frames consisting of eigenfunctions of localization operators. We only stated theorem \ref{ellipse_invariance} for the dilated Gaussian window, but the result holds for all dilated Hermite functions since they are eigenfunctions of the quadratic Fourier transform and have eigenvalues of modulus 1.

Let $H^T H = \frac{2}{\sqrt{3}}
\left(
	\begin{array}{rr}
	1 & \pm \frac{1}{2}\\
	\pm \frac{1}{2} & 1
	\end{array}
\right)$, such that $H$ generates some version of a rotated hexagonal lattice of volume $1$. Taking the standard Gaussian $g_1$ as window function, any rotated version of the hexagonal lattice gives the same frame bounds, hence, we may choose one representative among all versions of the rotated hexagonal lattice and denote its generator matrix by $H_0$. Let $\delta > 1$ be the density of the lattice such that $\G(g_1, \frac{1}{\sqrt{\delta}} H_0 \Z^2)$ is a Gabor frame. Then, using theorem \ref{ellipse_invariance}, we gain a result closely related to the problem of optimal pulse shape design for LOFDM \cite{StrBea03}. The matrix $S = 
\left(
	\begin{array}{cc}
	\sqrt{2} & \frac{1}{\sqrt{2}}\\
	0 & \frac{1}{\sqrt{2}}
	\end{array}
\right)$ generates a lattice which is a $45$ degrees rotated version of the integer lattice. If we want
\begin{equation}
	\G\left(g_1,\frac{1}{\sqrt{\delta}} H_0 \Z^2\right) \cong \G\left(g,\frac{1}{\sqrt{\delta}} S \Z^2\right)
\end{equation}
for some $g \in \Lt[]$, then
\begin{equation}
	Ag(x,\omega) = e^{-\frac{\pi}{2} \left(\sqrt{3} x^2 + \frac{\omega^2}{\sqrt{3}}\right)}.
\end{equation}
So, the quadratic form in the exponent of the ambiguity function describes an ellipse in standard position with semi-axis ratio $\sqrt{3}$. Furthermore, we know that there are uncountably many other lattice arrangements which together with $g$ lead to the same frame bounds, namely
\begin{equation} 
	\G\left(g,\frac{1}{\sqrt{\delta}} S \Z^2\right) \cong \G\left(g,\frac{1}{\sqrt{\delta}}S_{\tau, m}  S \Z^2\right).
\end{equation}
This result carries over to and extends the results given in \cite{StrBea03}. We will also discuss this result in Example \ref{ex_hex_sqrt3}.

\subsection{Modular Deformations of Gabor Frames}
In this section we will deal with discrete deformations of Gabor frames. In particular, the objects of interest are taken from the modular group which we define as follows.
\begin{definition}\label{definition_modular_group}
	The \textit{modular group} $Sp(2,\Z)$ consists of all 2-dimensional symplectic matrices with integer entries.
\end{definition}
\begin{remark}
	Usually the modular group $\Gamma$ is defined as the group of linear fractional transformations on the upper half of the complex plane which have the form
	\begin{equation}
		z \mapsto \frac{az+b}{cz+d},
	\end{equation}
	with $a,b,c,d \in \Z$ and $ad-bc = 1$. For more details on the modular group see \cite{SteSha03}.
\end{remark}
Consider the integer lattice $\Z^2$. The action of the modular group leaves $\Z^2$ invariant, i.e.\ $\B\Z^2 = \Z^2$ for $\B \in Sp(2,\Z)$. In other words, $\B$ is just another choice for a basis of $\Z^2$. In particular, any $\B \in Sp(2,\Z)$ provides a basis for $\Z^2$. Taking any symplectic matrix $S \in Sp(2,\R)$ and any basis $\B \in Sp(2,\Z)$ for $\Z^2$ this implies that
\begin{equation}\label{eq_lattice_basis}
	S \Z^2 = S \B \Z^2.
\end{equation}
We stay with the integer lattice for the beginning. Let $\G(g,\frac{1}{\sqrt{\delta}} \Z^2)$, $\delta > 1$ be a Gabor frame and let
\begin{equation}
	\B = \left(
		\begin{array}{cc}
			a & b\\
			c & d
		\end{array}
	\right) \in Sp(2,\Z).
\end{equation}
The corresponding metaplectic operator is given by
\begin{equation}
	\widehat{\B}g(t) = i^{-\frac{1}{2}} \sqrt{\frac{1}{|b|}} \int_\R e^{2 \pi i W(t,t')} g(t') \, dt',
\end{equation}
where $W(t,t') = \frac{1}{2}\frac{d}{b}t^2 -\frac{1}{b}tt' + \frac{1}{2}\frac{a}{b}t'^2$ and $b \neq 0$. In general $\widehat{B}g$ will differ from $g$ by more than just a phase factor as we apply a chirp a modified Fourier transform a dilation and again a chirp, but the lattice $\Lambda_{I,\delta} = \frac{1}{\sqrt{\delta}} \Z^2$ remains invariant under a modular deformation. Hence,
\begin{equation}
	\G(g,\Lambda_{I,\delta}) \cong \G(\widehat{\B}g,\Lambda_{I,\delta}).
\end{equation}
This result can be extended in an obvious way. Let $S \in Sp(2,\R)$ and let $\metaop[] \in Mp(2,\R)$ be the corresponding metaplectic operator, then,
\begin{equation}
	\G(\metaop[] g,S \Lambda_{I,\delta}) \cong \G(\metaop[] \widehat{\B}g, S \Lambda_{I,\delta}).
\end{equation}
Therefore, given any lattice $\Lambda = S \Lambda_{I,\delta}$ there are countably many possible windows which lead to the same Gabor frame bounds. We sum up the results in the following theorem.
\begin{theorem}
	Let $S \in Sp(2,\R)$, $\B \in Sp(2,\Z)$ and let $\metaop[]$ and $\widehat{\B}$ be the corresponding metaplectic operators. Let $\Lambda_{I,\delta} = \frac{1}{\sqrt{\delta}} \Z^2$ with $\delta > 1$ and let $g$ be a window function. Then
	\begin{equation}
		\G(\metaop[] g,S \Lambda_{I,\delta}) \cong \G(\metaop[] \widehat{\B}g, S \Lambda_{I,\delta}).
	\end{equation}
\end{theorem}
\begin{remark}
	Whereas the deformations in the previous section have been derived from a continuous, compact group, the deformations in this section are derived from a discrete, non-compact group. Continuous deformation groups will in general change the lattice, whereas the window might stay invariant under the corresponding deformation. Discrete deformation groups will in general change the window, whereas the lattice might stay invariant under the corresponding deformation.
\end{remark}

\subsection{Examples for Generalized Gaussians}

We will now illustrate our intuitive geometric approach by example. We will use different generalized Gaussians and different lattices.
\begin{example}\label{ex_hex_sqrt3}
	We start with an example inspired by \cite{StrBea03}. Let 
	\begin{equation}
		\Lambda_H = \frac{1}{\sqrt{\delta}} \sqrt{\frac{2}{\sqrt{3}}}
		\left(
			\begin{array}{rc}
				\cos(\pi/6) & \cos(\pi/6)\\
				-\sin(\pi/6) & \sin(\pi/6)
			\end{array}
		\right)
	\end{equation}
	be a version of the hexagonal lattice of density $\delta > 1$. Since, $\Lambda_H$ is radial symmetric, we choose the standard Gaussian $g_1$ as window function and have the Gabor system $\mathcal{G}(g_1,\Lambda_H)$ which is a Gabor frame. We apply the dilation matrix $M_{3^{-1/4}}$ on the lattice and the rescaling operator $\rescaleop[3^{-1/4}]$ on the window. By theorem \ref{ellipse_invariance} we know that
	\begin{equation}
		\mathcal{G}(g_1,\Lambda_H) \cong \mathcal{G}(\rescaleop[3^{-1/4}]g_1, M_{3^{-1/4}} \Lambda_H).
	\end{equation}
	Furthermore, we compute
	\begin{equation}
		M_{3^{-1/4}} \Lambda_H = \frac{1}{\sqrt{\delta}} \RotMat[(\pi/4)] \Z^2 = \frac{1}{\sqrt{\delta}} S_{\frac{\pi}{4}} \Z^2,
	\end{equation}
	which is a $45$ degrees rotated version of the integer lattice scaled to have density $\delta > 1$.
	\begin{figure}[ht]
		\centering
		\subfigure[$\tau=\pi/12$]{\includegraphics[width=.45\textwidth]{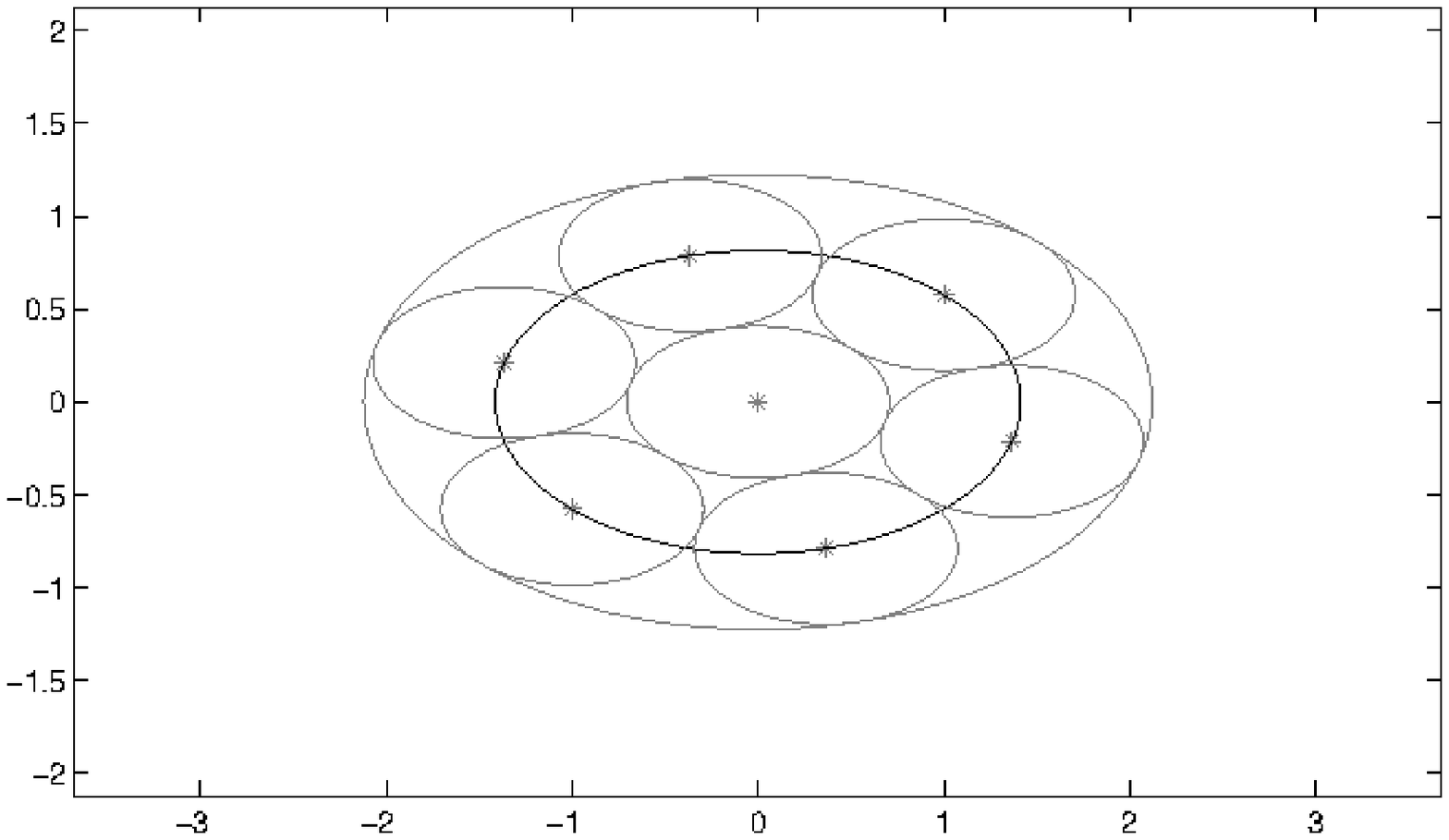}}
		\hfill
		\subfigure[$\tau=0$]{\includegraphics[width=.45\textwidth]{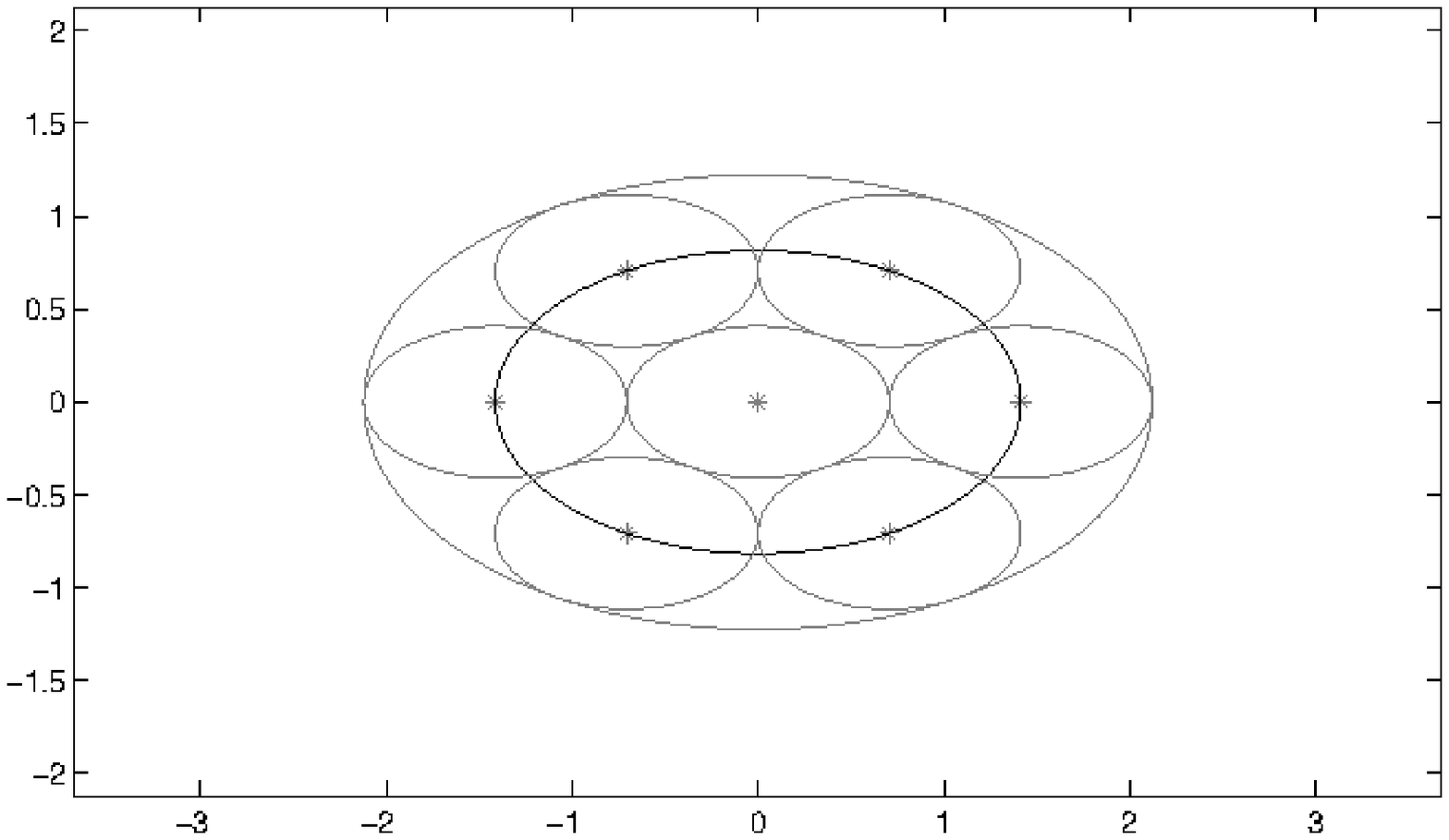}}\\
		\subfigure[$\tau=-\pi/12$]{\includegraphics[width=.45\textwidth]{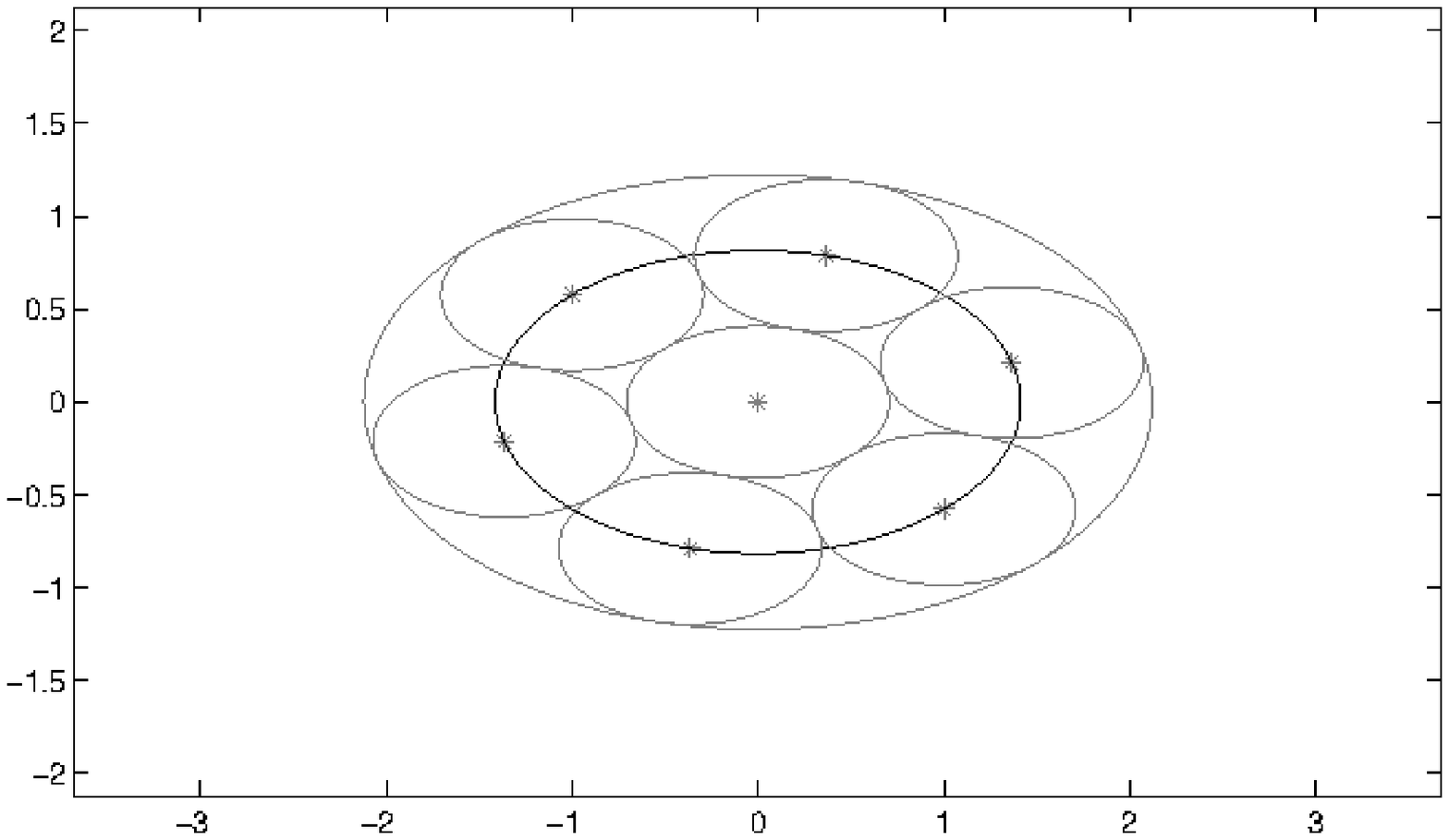}}
		\hfill
		\subfigure[$\tau=-\pi/6$]{\includegraphics[width=.45\textwidth]{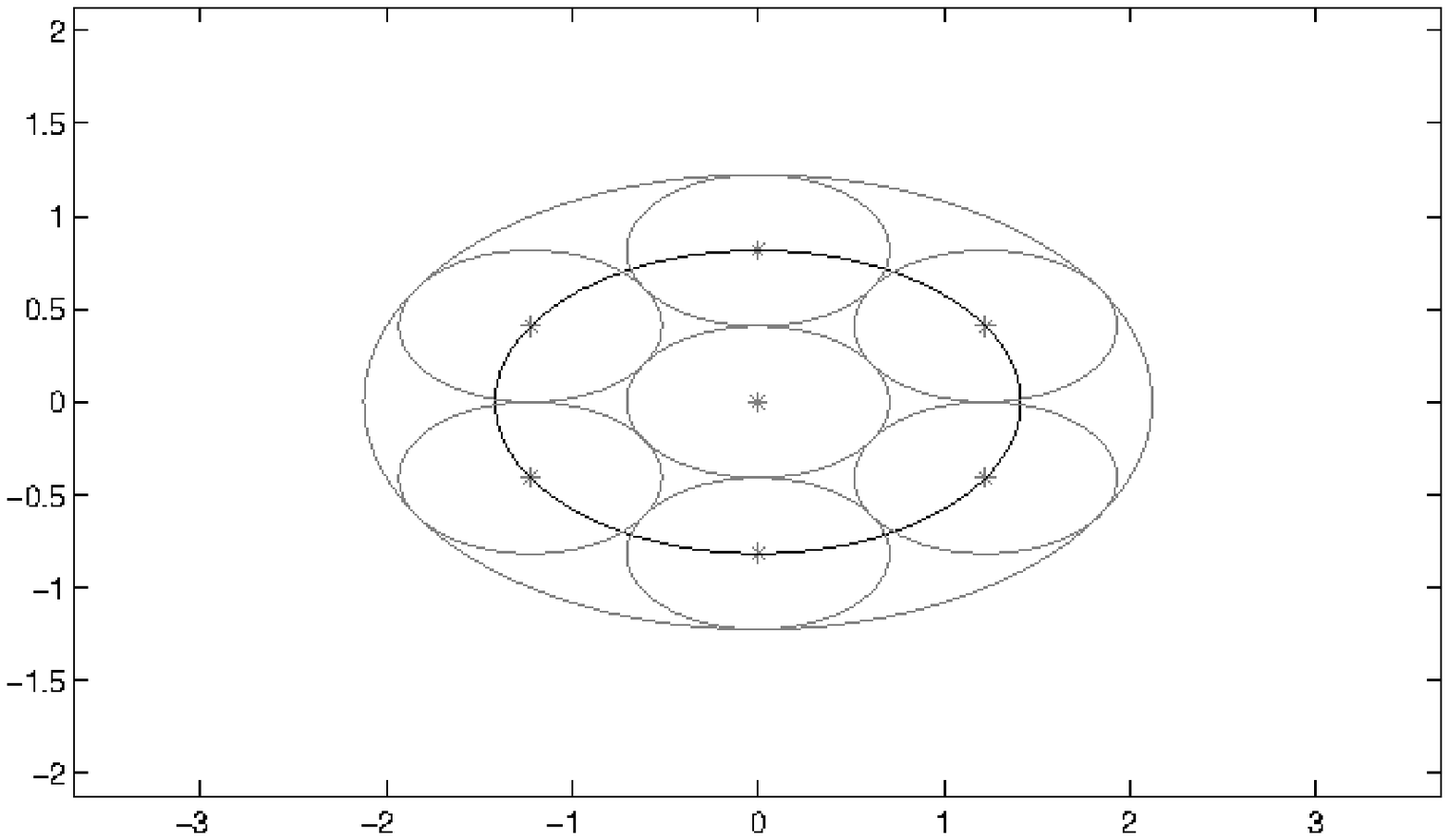}}
		\caption{Illustration of the action of $S_{\tau,m}$ on the lattice and of $\metaop[\tau,m]$ on the ambiguity function. The small ellipses illustrate the ambiguity functions centered at lattice points. The ellipses centered at the origin indicate flow lines of the harmonic oscillator.}\label{fig_Hamilton_deformation}
	\end{figure}
		
	The ambiguity function of $g_1$ is given by
	\begin{equation}
		Ag_1(x,\omega) = e^{-\frac{\pi}{2}(x^2+\omega^2)}
	\end{equation}
	and the ambiguity function of $\rescaleop[3^{-1/4}]g_1 = g_{\sqrt{3}}$ is given by
	\begin{equation}
		Ag_{\frac{1}{\sqrt{3}}}(x,\omega) = e^{-\frac{\pi}{2}\left(\frac{x^2}{\sqrt{3}} + \sqrt{3} \omega^2\right)}.
	\end{equation}
	Applying the matrix
	\begin{equation}
		S_{\tau,\frac{1}{\sqrt{3}}} =
		\left(
			\begin{array}{rc}
				\cos \tau & \sqrt{3} \sin \tau\\
				-\frac{\sin \tau}{\sqrt{3}} & \cos \tau
			\end{array}
		\right),
	\end{equation}
	derived from the flow of the harmonic oscillator with mass $m=\frac{1}{\sqrt{3}}$ on the lattice will leave the frame bounds unchanged and we have
	\begin{equation}
		\mathcal{G}\left(g_1,\Lambda_H \right) \cong \mathcal{G}\left(g_{\frac{1}{\sqrt{3}}},\frac{1}{\sqrt{d}}S_{\frac{\pi}{4}}\Z^2\right) \cong \mathcal{G}\left(g_{\frac{1}{\sqrt{3}}},\frac{1}{\sqrt{d}}S_{\tau,\frac{1}{\sqrt{3}}} \, S_{\frac{\pi}{4}}\Z^2\right).
	\end{equation}
	The deformation process is illustrated in the Figure \ref{fig_Hamilton_deformation}.
\end{example}
\begin{example}\label{ex_modular_go}
	Let $\Lambda_{I,\delta} = \frac{1}{\sqrt{\delta}} \Z^2$ be the scaled integer lattice of density $\delta > 1$ and let $g_1(t) = \go$ be the standard Gaussian. The standard symplectic form $J$ belongs to the modular group $Sp(2,\Z)$. Hence, $J \Lambda_{I,\delta} = \Lambda_{I,\delta}$ and $\modFT g_1 = c \, g_1$ with $|c| = 1$. In this case neither the change of basis, nor the metaplectic operation have an effect on the Gabor frame. This is due to the fact that we could interpret the change of basis as a rotation of the time-frequency plane. Using the quadratic representation of $g_1$, the ambiguity function, we see that a rotation does not have an effect since it reads $Ag_1(x,\omega) = e^{-\frac{\pi}{2}(x^2+\omega^2)}$.
	
	Let us now consider the case of $g_{\sqrt{3}}$ and $\Lambda = \frac{1}{\sqrt{\delta}} S_{\frac{\pi}{4}} \Z^2$ which is a $90$ degrees rotated version compared to the window in Example \ref{ex_hex_sqrt3} which can also be interpreted as the deformation of the window under a change of basis using $J$ as basis. We know the ambiguity function of $Ag_{\sqrt{3}}(x,\omega) = e^{-\frac{\pi}{2} \left( \sqrt{3} x^2+\frac{\omega^2}{\sqrt{3}} \right)}$. We rotate our lattice by the matrix $S_{-\frac{\pi}{4}}$ and apply the corresponding operator $\metaop[-\frac{\pi}{4}]$ on the window. Hence, the ambiguity function of the new window becomes
	\begin{equation}
		\begin{aligned}
			A \metaop[-\frac{\pi}{4}] g_{\sqrt{3}}(x,\omega) & =  Ag_{\sqrt{3}} \left( S_{-\frac{\pi}{4}}^{-1} \lambda \right)\\
			& =  Ag_{1} \left(M^{-1}_{3^{1/4}} S_{-\frac{\pi}{4}}^{-1} \lambda \right)\\
			& = e^{-\frac{\pi}{2} \langle \left(S_{-\frac{\pi}{4}} M_{3^{1/4}}\right)^{-1} \lambda, \, \left(S_{-\frac{\pi}{4}} M_{3^{1/4}}\right)^{-1} \lambda \rangle}\\
			& = e^{-\frac{\pi}{2} \left(M_{3^{-1/4}} S_{\frac{\pi}{4}} \lambda \right)^T \cdot \left(M_{3^{-1/4}} S_{\frac{\pi}{4}} \lambda \right)}\\
			& = e^{-\frac{\pi}{2} \, \lambda^T S_{\frac{\pi}{4}}^T M_{3^{-1/4}}^T M_{3^{-1/4}} S_{\frac{\pi}{4}} \lambda}\\
			& = e^{-\frac{\pi}{2} \, \lambda^T S_{-\frac{\pi}{4}} M_{3^{-1/2}} S_{\frac{\pi}{4}} \lambda}\\
			& = e^{-\frac{\pi}{2} \frac{2}{\sqrt{3}} \left(x^2 + x \omega + \omega^2 \right)}.
		\end{aligned}
	\end{equation}
	Therefore, the level lines of the ambiguity function will be ellipses rotated by $45$ degrees with semi-axis ratio equal to $\sqrt{3}$ and the lattice will be a scaled version of the integer lattice with density $\delta > 1$. The action of the metaplectic operator can be interpreted in a very natural and geometric way as can be seen by the calculations above.
	
	\begin{figure}[ht]
		\centering
		\subfigure[$A \widehat{S_{-\frac{\pi}{4}}} g_{\sqrt{3}}$]{\includegraphics[width=.45\textwidth]{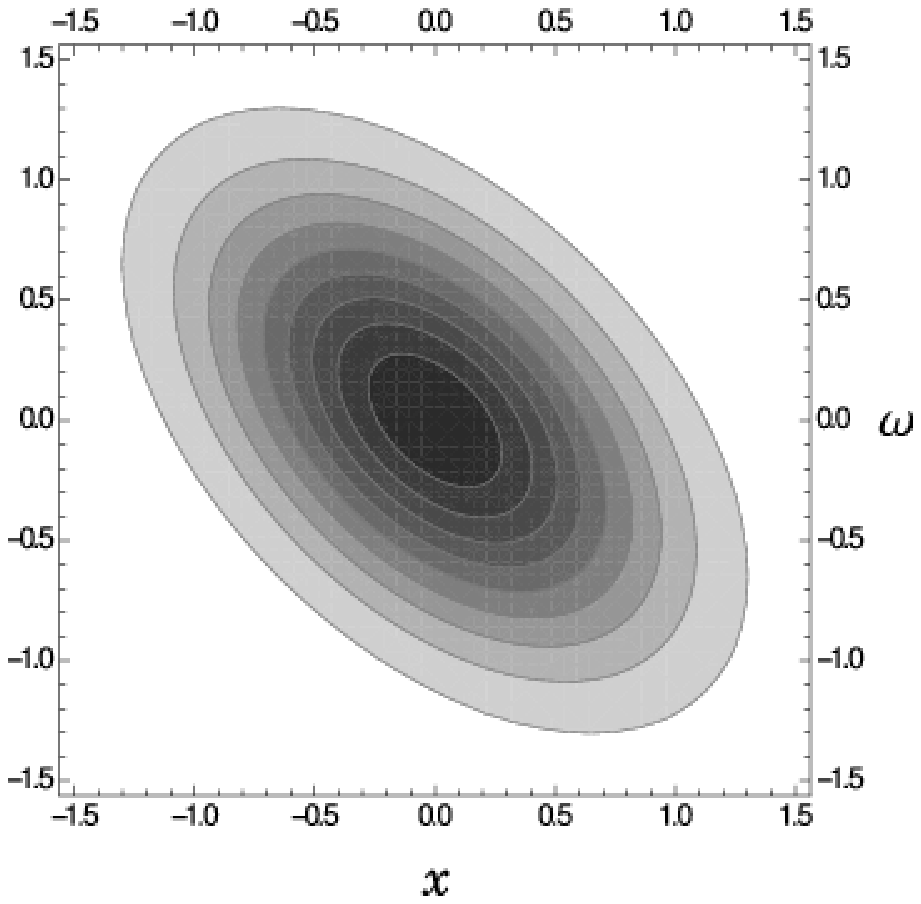}}
		\hfill
		\subfigure[$A \widehat{S_{-\frac{\pi}{4}}} g_{\frac{1}{\sqrt{3}}}$]{\includegraphics[width=.45\textwidth]{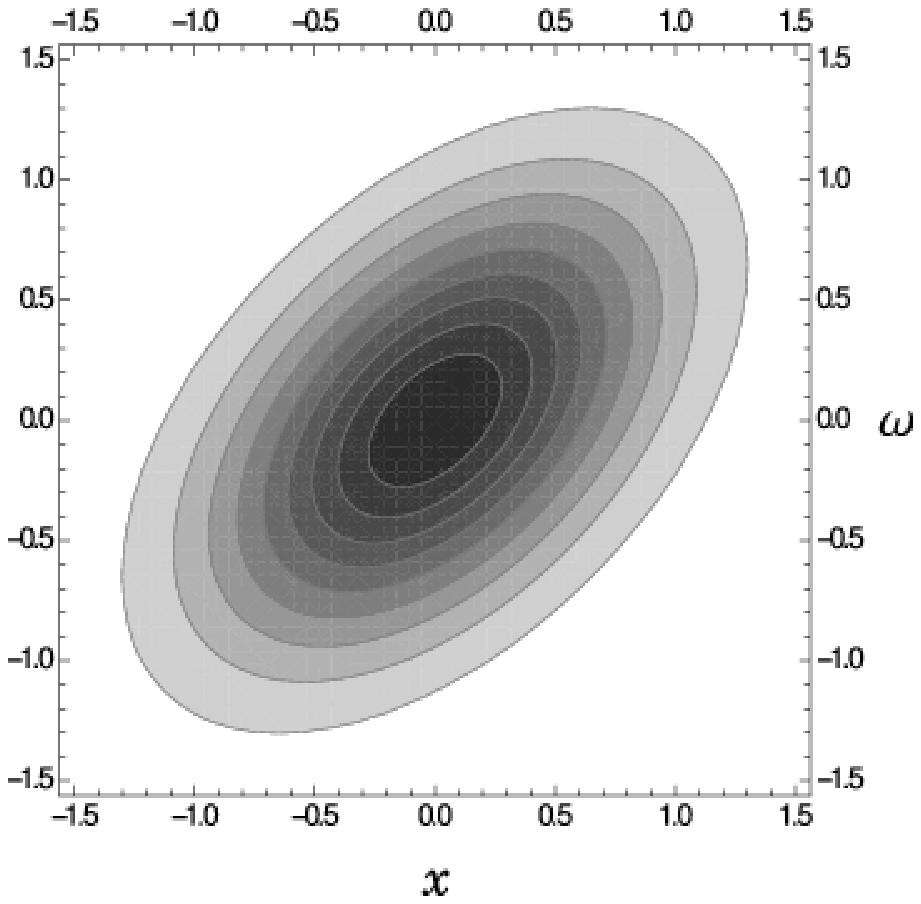}}
		\caption{Contour plots of the ambiguity functions of two possible generalized Gaussians which lead to the same frame bounds for the scaled integer lattice of density $\delta > 1$.}\label{fig_Modular_Deformation}
	\end{figure}
	
	We could also have rotated the lattice $\Lambda = \frac{1}{\sqrt{\delta}} S_{\frac{\pi}{4}} \Z^2$ in the opposite direction by $S_{\frac{\pi}{4}}$. Then, we would have had $\frac{1}{\sqrt{\delta}}J \Z^2$ as our lattice, which leads basically to the same lattice, but with another choice of basis. Hence, the window as well as the ambiguity function would have changed. In this case, we can interpret the deformation as a rotation, a shearing or a choice of a different basis. The interpretation is left to the reader, but the ambiguity function will be
	\begin{equation}
		A \metaop[-\frac{\pi}{4}]^{-1} g_{\sqrt{3}} (x,\omega) = A \metaop[-\frac{\pi}{4}] g_{1/\sqrt{3}} (x, \omega) = e^{-\frac{\pi}{2} \frac{2}{\sqrt{3}} \left(x^2 - x \omega + \omega^2 \right)}.
	\end{equation}
	The calculations are left to the reader, but the geometric understanding of the lattice deformation is sufficient to understand the deformation of the ambiguity function. The level lines describe the same ellipses as before only rotated by $90$ degrees (see Figure \ref{fig_Modular_Deformation}).
\end{example}

\begin{remark}
	Instead of the ambiguity function, which is commonly used in time-frequency analysis, we could as well have used the Wigner distribution (see appendix \ref{App_AmbiguityFunction}). The Hamiltonian formulation of the harmonic oscillator describes the behaviour of a system in classical mechanics which can also be formulated as a problem in  quantum mechanics. The Hermite functions are eigenstates of the propagator which corresponds to the flow of the classical problem. Hence, the deformation of the lattice corresponds to a problem in classical mechanics and the deformation of the window corresponds to the same problem formulated in terms of quantum mechanics. Therefore, using the Wigner distribution, all results can also be interpreted for problems in classical and quantum mechanics.
\end{remark}

\begin{appendices}
\section{The Ambiguity Function}\label{App_AmbiguityFunction}

\begin{definition}\label{AmbiguityFunction}
	The \textit{ambiguity function} of a function $f \in \Lt$ is given by
	\begin{equation}
		Af(x,\omega) = \int_{\R^d} f\left(t+\frac{x}{2}\right) \overline{f\left(t-\frac{x}{2}\right)} e^{-2 \pi i \omega \cdot t} \, dt.
	\end{equation}
	In a similar way we define the \textit{cross-ambiguity function} of two functions $f,g \in \Lt$
	\begin{equation}
		A_g f(x,\omega) = \int_{\R^d} f\left(t+\frac{x}{2}\right) \overline{g\left(t-\frac{x}{2}\right)} e^{-2 \pi i \omega \cdot t} \, dt.
	\end{equation}
\end{definition}
The cross-ambiguity function and hence, just as well the ambiguity function are closely related to the \textit{short-time Fourier transform}
\begin{equation}
	\mathcal{V}_g f(x,\omega) = \langle f, \pi(\lambda) g \rangle = \int_{\R^d} f(t) \overline{g(t-x)} e^{-2 \pi i \omega \cdot t} \, dt.
\end{equation}
In fact, they only differ by a phase factor and we have
\begin{equation}
	A_gf(x, \omega) = e^{\pi i x \cdot \omega} \mathcal{V}_g f(x, \omega).
\end{equation}
The difference in the phase factor is due to the fact that the translation and modulation operators do not commute.
\begin{equation}
	M_\omega T_x = e^{2 \pi i x \cdot \omega} T_x M_\omega,
\end{equation}
therefore, we have
\begin{equation}
	\begin{aligned}
		\mathcal{V}_g f(x,\omega) & = \langle f, M_\omega T_x g \rangle\\
		& = \langle T_{-x/2} M_{-\omega} f, T_{x/2} g \rangle\\
		& = e^{-\pi i x \cdot \omega} \langle M_{-\omega} T_{-x/2} f, T_{x/2} g \rangle\\
		& = e^{-\pi i x \cdot \omega} \langle M_{-\omega/2} T_{-x/2} f, M_{\omega/2} T_{x/2} g \rangle\\
		& = e^{-\pi i x \cdot \omega} \langle \pi(-\lambda/2) f, \pi(\lambda/2) g \rangle\\
		& = e^{-\pi i x \cdot \omega} A_g f(x, \omega).
	\end{aligned}
\end{equation}
The ambiguity function is somehow a more symmetric time-frequency representation of a signal than the short-time Fourier transform. In addition, the ambiguity function of a dilated Gaussian is real-valued as we have seen in \eqref{eq_Ambiguity_generalized_Gaussian}. Furthermore, the ambiguity function already determines a function up to a phase factor \cite{Gro01}. The usual interpretation of the ambiguity function is that it tells how much a function is spread in time and frequency, similar to the interpretation of the \textit{Wigner distribution} in physics. The Wigner distribution is given by
\begin{equation}
	\mathcal{W}f(x,\omega) = \int_{\R^d} f\left(x+\frac{t}{2}\right) \overline{f\left(x-\frac{t}{2}\right)} e^{-2 \pi i \omega \cdot t} \, dt.
\end{equation}
It is related to the ambiguity function by the symplectic Fourier transform
\begin{equation}
	\mathcal{W} f (\lambda) = \mathcal{F} (A f) (J \lambda),
\end{equation}
where $\mathcal{F}f(\omega) = \int_{\R^{d}} f(t) e^{-2 \pi i \omega \cdot t} \, dt$ is the Fourier transform on $\Lt[d]$.

\section{Proof of Theorem \ref{ellipse_invariance}}\label{proof_ellipse_invariance}

We will now prove equations \eqref{eq_ellipse_invariance_window_thm} and \eqref{eq_ellipse_invariance_ambiguity_thm} of theorem \ref{ellipse_invariance}. Since the ambiguity function already determines a function up to a factor of modulus 1, equation \eqref{eq_ellipse_invariance_window_thm} and \eqref{eq_ellipse_invariance_ambiguity_thm} are equivalent. First, we note that
\begin{equation}
	S_{\tau,m} \, M_{\sqrt{m}} =
	\left(
		\begin{array}{cc}
			\frac{\cos \tau}{\sqrt{m}} & \frac{\sin \tau}{\sqrt{m}}\\
			-\sqrt{m} \sin \tau & \sqrt{m} \cos \tau
		\end{array}
	\right) = M_{\sqrt{m}} S_{\tau},
\end{equation}
where $S_{\tau}=S_{\tau,1}$ is a rotation by $-\tau$. This means, that imposing the elliptic flow $S_{\tau,m}$ on the dilated lattice is the same as rotating the lattice by the corresponding angle followed by the same dilation. The second ingredient we need in order to perform the proof is the covariance principle
\begin{equation}
	\pi(\lambda) \metaop[] = \metaop[] \pi\left(S^{-1} \lambda\right),
\end{equation}
where $\pi(\l) = M_\omega T_x$ is the time-frequency shift as defined in \eqref{eq_TFShift}.
This implies that
\begin{equation}
	A \left(\metaop[]f\right)(\lambda) = A f(S^{-1} \lambda).
\end{equation}
This is a classical result \cite{Fol89} which is also used in \cite{StrBea03} and a similar result for the STFT is given in \cite{Gro01}. Hence, using the fact that
\begin{equation}
	Ag_1(\lambda) = Ag_1(x,\omega) = e^{-\frac{\pi}{2} (x^2+\omega^2)} = e^{-\frac{\pi}{2} \langle \lambda, \; \lambda \rangle}
\end{equation}
we compute
\begin{equation}
	\begin{aligned}
		A \left( \metaop[\tau,m] \, g_{m} \right) (x,\omega) & = A \left( \metaop[\tau,m] \rescaleop[\sqrt{m}] \, g_1 \right) (x,\omega)\\
		& = A g_1 \left(\left(S_{\tau,m} M_{\sqrt{m}}\right)^{-1} \lambda\right)\\
		& = e^{-\frac{\pi}{2} \langle \left(S_{\tau,m} \, M_{\sqrt{m}}\right)^{-1} \lambda, \; \left(S_{\tau,m} \, M_{\sqrt{m}}\right)^{-1} \lambda \rangle}\\
		& = e^{-\frac{\pi}{2} \langle \left(M_{\sqrt{m}} \, S_{\tau} \right)^{-1} \lambda, \; \left(M_{\sqrt{m}} \, S_{\tau}\right)^{-1} \lambda \rangle}\\
		& = e^{-\frac{\pi}{2} \langle S_{\tau}^{-1} M_{\sqrt{m}}^{-1} \; \lambda, \; S_{\tau}^{-1} M_{\sqrt{m}}^{-1} \; \lambda \rangle}\\
		& = e^{-\frac{\pi}{2} \langle M_{\sqrt{m}}^{-1} \; \lambda, \; M_{\sqrt{m}}^{-1} \; \lambda \rangle}\\
		& = e^{-\frac{\pi}{2} \left(m x^2 + \frac{\omega^2}{m}\right)}\\
		& = Ag_{m}(x,\omega).
	\end{aligned}
\end{equation}

This proves equation \eqref{eq_ellipse_invariance_ambiguity_thm} and hence, the proof of theorem \ref{ellipse_invariance} is complete.
\end{appendices}

\subsubsection*{Acknowledgements}
	The author wishes to thank Maurice de Gosson and Karlheinz Gr\"ochenig for many fruitful discussions on the topic. The author wishes to thank the reviewers for helpful comments to improve this work. The author was supported by the Austrian Science Fund (FWF): [P26273-N25].

\addcontentsline{toc}{part}{References}
\bibliographystyle{plain}
\bibliography{diss}

\end{document}